\documentclass{article}

\usepackage[numbers,square]{natbib}

\usepackage{graphicx}
\usepackage{amssymb,amsmath,amsthm}
\usepackage{bbm}
\usepackage{mhequ}
\usepackage{float}
\newcommand{\bone}[1]{\mathbbm{1}_{\{ #1 \}}}
\newcommand{\bb}[1]{\mathbb{#1}}

\newcommand{\be}{\begin{equs}}
\newcommand{\ee}{\end{equs}}
\newcommand{\wh}[1]{\widehat{#1}}
\newcommand{\mc}[1]{\mathcal{#1}}
\DeclareMathOperator{\No}{No}
\DeclareMathOperator{\SE}{SE}
\renewcommand{\P}{\mathbb P}

\DeclareMathOperator{\opt}{opt}
\DeclareMathOperator{\Ca}{C}
\DeclareMathOperator{\uniform}{U}

\theoremstyle{plain}
\newtheorem{definition}{Definition}
\newtheorem{theorem}{Theorem}

\graphicspath{{./}}

\begin{document}

\title{A Decision Theoretic Approach to A/B Testing}
\author{ David Goldberg \\ eBay \and James E. Johndrow \\ Stanford University}

\maketitle

\begin{abstract}
A/B testing is ubiquitous within the machine learning and data science 
operations of internet companies. Generically, the idea is to perform a 
statistical test of the hypothesis that a new feature is better than the 
existing platform---for example, it results in higher revenue. If the p value 
for the test is below some pre-defined threshold---often, 0.05---the new feature 
is implemented. The difficulty of choosing an appropriate threshold has 
been noted before, particularly because dependent tests are often done sequentially, leading 
some to propose control of the false discovery 
rate (FDR) \cite{pekelis2015new,correia2016bayesian,javanmard2016online} rather than use of a single, universal threshold. 
However, it is still necessary to make an arbitrary choice of the level at which to control FDR.
Here we 
suggest a decision-theoretic approach to determining whether to adopt a new 
feature, which enables automated selection of an appropriate threshold. Our method has the basic 
ingredients of any decision-theory problem: a loss function, action space, and a 
notion of optimality, for which we choose Bayes risk. However, the loss function 
and the action space differ from the typical choices made in the literature, 
which has focused on the theory of point estimation. We give some basic results 
for Bayes-optimal thresholding rules for the feature adoption decision, and 
give some examples using eBay data. The results suggest that the 0.05 p-value 
threshold may be too conservative in some settings, but that its widespread use 
may reflect an ad-hoc means of controlling multiplicity in the common case of repeatedly 
testing variants of an experiment when the threshold is not reached. 
\end{abstract}

\section{Basic A/B Testing Problem}
In A/B testing, one has a proposed new version of a software platform and wants to 
decide whether or not to ship the new version. The classical way of conceiving 
of this problem is the following. We divide users into two groups: treatment and 
control. We then roll out the proposed update to the treatment group while 
leaving the control group with the current version. Using data gathered from 
this randomized trial, we then ask whether the new version performed 
``better'' with respect to some metric. For the purposes of grounding the 
discussion, we assume that the metric is revenue, which at eBay is 
roughly equivalent to Gross Merchandise Bought (GMB).

The  literature on A/B testing has considered several aspects of this problem, 
ranging from sequential testing issues \cite{johari2015always,deng2016continuous,deng2015objective}, to study of 
multi-armed bandits that approximately characterize some applications like search 
engine optimization and page customization 
\cite{scott2010modern,scott2015multi,perchet2013multi}, to practical and 
computational issues 
\cite{kohavi2009controlled,kohavi2013online,kohavi2007practical}. In e-commerce, 
tests often need to run for several weeks, so it is usually not practical to 
keep multiple competing versions of the platform active over a period of time in 
order to pursue an explore and exploit strategy. Accordingly, the traditional 
A/B testing framework in which a decision is made after every experiment is favored. 
Sequential testing and multilevel hierarchical dependence structures among experiments
are issues in experimentation at eBay, and we return to this in Section \ref{sec:Hierarchical}.

The traditional or generic view is to treat the feature adoption decision like 
a one-sided hypothesis testing problem. The relevant hypothesis is
\be \label{eq:hypothesis}
H_0: \text{current version is better}.
\ee
A  very simple setup in which to consider this is to let $\theta_0$ be the 
revenue per user in the control group and $\theta_1$ the revenue per user in the 
treatment group, so that %$\Delta = \theta_1-\theta_0$ is the treatment effect, 
%which is 
%sometimes expressed as \emph{lift}, the percentage change from the control 
%group, 
so that $\Delta = 100(\theta_1 - \theta_0)/\theta_0$ is the \emph{lift}, the 
percentage change in revenue relative to the control group. On observing 
data 
\be
x \sim F(x \mid \theta)
\ee
from the treatment ($\theta = \theta_0$) and control ($\theta = \theta_1$) 
groups, we use some procedure, which we leave abstract, to obtain an 
estimate $\wh \Delta$. We then assume that, at least approximately,
\be
\frac{\wh{\Delta}-\Delta}{\SE(\wh{\Delta})} \sim t_{\nu},
\ee
where $t_\nu$ is a $t$ distribution with $\nu$ degrees of freedom, $\SE(\wh 
\Delta)$ is the standard error of $\wh \Delta$, and $\nu$ is known. Now, letting
\be
T(x) = \frac{\wh{\Delta}}{\SE(\wh \Delta)}
\ee
we compute the tail probability under repeated sampling
\be \label{eq:pvalue}
p = \bb P[T(X) > T(x) \mid H_0],
\ee
the one-tailed $p$ value. We then threshold the $p$ value at some level -- typically, 
0.05 -- and decide to ship if the $p$ value is smaller than the threshold.

\section{A Decision Theoretic Perspective}
An alternative way to look at A/B testing is as a decision theory problem
rather than an inference problem. That is, our primary goal is not to validate 
or invalidate the scientific hypothesis in \eqref{eq:hypothesis}, but to maximize revenue for the company.
%we do not care about the 
%hypothesis in \eqref{eq:hypothesis} \emph{per se}, we care about making money. 
In decision theory, we have an \emph{action space} $\mc A$ consisting of all of 
the possible decisions we can make, and a \emph{loss function} $L(\theta,a)$ 
which defines what we lose if the true state of nature is $\theta$ and we decide 
to take action $a \in \mc A$. For A/B testing, the action space only has two 
elements: ``ship'' and ``don't ship.'' The obvious loss function is
\be \label{eq:loss}
L(\Delta,a) = -a \Delta
\ee
where $a = \bone{\textnormal{ship}}$.  That is, if we choose to ship and the 
true lift is positive, then we gain the lift (equivalently, we lose the negative of the 
lift). Otherwise we lose zero; we just get business as usual GMB. Note that our 
decision rule $a$ is something that we do upon observing data, so $a = 
\delta(x)$ is a map from the sample space $\mc X$ into $\mc A$. We emphasize 
that the loss function in \eqref{eq:loss} is unusual in the 
literature, which focuses on loss functions like squared error and continuous action space, and 
thus the results we derive here are somewhat nonstandard.

The aim of decision theory is to choose an optimal decision rule. The 
frequentist perspective on decision rule optimality is to compute the expected 
loss if we use $\delta(X)$ in repeat sampling.\footnote{Throughout we use the 
standard 
convention of denoting random variables by upper case Roman letters and their realizations
by lower case Roman letters.} This is called the risk
\begin{definition}[risk]
The \emph{risk function} of a decision rule $\delta$ is defined as
\be
R(\theta,\delta) = \bb E_{\theta}[L(\theta,\delta(X))] = \int 
L(\theta,\delta(x)) dF(x \mid \theta),
\ee
the expectation of the loss over the sampling distribution of the data 
conditioned on $\theta$.
\end{definition} 

The risk \emph{conditions on} $\theta$,  the unknown state of nature. 
Since $\theta$ is unknown, we seek a decision rule that performs well 
no matter the true value of $\theta$. There are several ways to formalize this. 
We focus on the \emph{Bayes risk}
%The central object in the Bayes paradigm is the \emph{Bayes risk}
\begin{definition}[Bayes risk]
The Bayes risk of a decision rule $\delta$ is defined as the prior expectation of the risk
\be
r(\pi,\delta) = \bb E_{\pi}[R(\theta,\delta)] = \bb E_{\pi} [\bb E_{F(x \mid \theta)}[L(\theta,\delta(x))]],
\ee
where $\bb E_\pi[f(\theta)] = \int f(\theta) \pi(d\theta)$ is the expectation of $f$ with respect to $\pi$. 
\end{definition}
A decision rule is considered Bayes optimal if it minimizes the Bayes risk. 
Thus, Bayes risk deals with the fact that $\theta$ is unknown by \emph{weighting the states of nature} by 
our prior beliefs about their plausibility. In the applications that follow, we 
will take an empirical Bayes approach, where we estimate $\pi$ from the data.

With this basic idea in hand, we can consider the set of all decision rules for  
A/B testing that correspond to thresholding a $p$-value and derive the risk 
function. Suppose the true value of the lift is $\Delta$ and we define 
$\delta(x) = \bone{p(x) < \alpha}$. To simplify calculations, we initially consider a 
simpler version of the A/B testing problem, where $x$ is a noisy observation of 
the unknown lift with known variance
\be \label{eq:ObsData}
x \sim \No(\Delta,\sigma^2),
\ee
in lieu of the $t$ distribution, which arises when $\sigma^2$ is unknown. In this case
\be
T(x) = \frac{x}{\sigma} \sim \No\left( \frac{\Delta}{\sigma},1 \right),
\ee
the $p$ value is
\be
p(x) = \bb P[T(X) > T(x) \mid H_0] = 1-\Phi(x/\sigma),
\ee
and the decision rule is given by $\delta(x) = \bone{p(x) < \alpha}$, 
%Thresholding $p(x)$ turns out to be equivalent to thresholding $T(x)$, 
so we can redefine
\be \label{eq:ThreshDelta}
\delta(x) = \bone{\frac{x}{\sigma} > \frac{\beta}{\sigma}} =  \bone{x>\beta}
\ee
and compute the risk in the Gaussian case as
\be
R(\Delta, \delta) &= - \Delta \int \bone{x > \beta} dF(x \mid \Delta) \\
&= - \Delta \int \bone{\frac{x-\Delta}{\sigma} > \frac{\beta-\Delta}{\sigma}} dF(x \mid \Delta) \\
&= - \Delta \P_{F(x \mid \Delta)}\left[ \frac{X-\Delta}{\sigma} > \frac{\beta-\Delta}{\sigma} \right] \\
&= - \Delta \left[1-\Phi\left( \frac{\beta-\Delta}{\sigma} \right) \right] \\
&= - \Delta \Phi\left( \frac{\Delta-\beta}{\sigma} \right).
\ee

In fact, if \eqref{eq:ObsData} had been any location-scale family with a density symmetric about the location, 
\be \label{eq:LSGeneral}
x \sim F(x;\Delta,\sigma)
\ee
we would have obtained
\be \label{eq:LSRisk}
R(\Delta, \delta) &= - \Delta F\left( \frac{\Delta-\beta}{\sigma} \right),
\ee
with $F$ the CDF of the member of the location-scale family with location 0 and scale 1. Moreover, if we had replaced $\frac{x}{\sigma}$ with the statistic 
\be
T(x) = \frac{\widehat \Delta}{\SE(\widehat \Delta)}, 
\ee
we would still have obtained this representation, since $T(x)$ has a $t$, which a location-scale family with a density that is symmetric about the location. We will therefore mainly consider the case where $F(x\mid \theta)$ is Gaussian, with the understanding that the approach extends to other location-scale families. 

Having defined the risk function, we consider the Bayes risk. Suppose the lifts are exchangeable realizations of a random variable, so that
\be \label{eq:Prior}
\Delta_i \stackrel{iid}{\sim} \pi(\Delta_i;\eta),
\ee
where $\eta$ are the prior hyperparameters. For example, $\pi$ could be a normal distribution with parameters $\eta = (\mu,\tau^2)$. A general expression for the Bayes risk of any thresholding decision rule is
\be
r(\pi,\delta) &= \bb E_{\pi} [ \bb E_{F(x \mid \Delta)}[ -\delta(x) \Delta]] \\
&= -\bb E_{\pi} [\Delta \bb E_{F(x \mid \Delta)}[\bone{x>\beta}]] \\
&= -\bb E_{\pi}[\Delta \bb P_{F(x \mid \Delta)}[X>\beta]] \\
&= -\bb E_{\pi}[\Delta (1-F(\beta \mid \Delta))],
\ee
%for some function $G : \bb R \to [0,1]$ satisfying $G(\Delta) = \bb P[\delta(x) = 1 \mid \Delta]$ for a $F(x \mid \Delta)$-measurable binary function $\delta$. 
In the sequel, we estimate $\pi$ from eBay data and obtain some explicit expressions for $r(\pi,\delta)$.

%A final note: our loss function and general setting are ``weird,'' that is, not heavily studied in the literature. Most of the theoretical analysis has focused on \emph{continuous} action spaces where e.g. our ``action'' is to estimate $\theta$ by $\hat{\theta}$, and the main loss functions considered are $\| \theta - \hat{\theta} \|_2$ (squared error loss) and $|\theta - \hat{\theta}|$ (absolute error loss). In this context, very strong results are known. In particular, it is know that the posterior expectation is the unique minimizer of the Bayes risk with squared error loss, and that the posterior median is the unique minimizer of the Bayes risk with absolute error loss. This suggests that ``good'' decision rules in our context may involve either thresholding of the posterior mean/median or using the probability of the mean/median exceeding zero to generate a randomized decision rule.

\section{Bayes Risk of Thresholding Rules}
We now return to the class of decision rules in \eqref{eq:ThreshDelta} that thresholds at $\beta$ a statistic that is distributed according to a location-scale family, with risk given in \eqref{eq:LSRisk}. If we knew $F$ in \eqref{eq:LSGeneral} and $\pi$ in \eqref{eq:Prior}, we could optimize the Bayes risk over $\beta$ to determine the Bayes-optimal strategy for deciding whether to ship a proposed update to the platform. In many applications, we are willing to assume that $F$ is $t$ or normal. This is particularly true in A/B testing applications in industry, where $x$ is typically the (normalized) difference of means from two quite large populations. At eBay, a somewhat more sophisticated procedure is used to estimate the lift, but the estimator is then assumed to be approximately Gaussian for hypothesis testing purposes. We make the same assumption here. We then model the true lifts as student $t$ with unknown location $\mu$, scale $\tau$, and degrees of freedom $\nu$ so that
\be \label{eq:model}
x_i \mid \Delta_i, \sigma^2_i &\sim \No(\Delta_i, \sigma^2_i), \\
\Delta_i &\stackrel{iid}{\sim} t_{\nu}(\mu,\tau)
\ee
%I suspect it won't matter too much what $F$ is so long as it is location-scale with location parameter $\Delta$. 
is a hierarchical Bayesian specification of the process generating the data, where we have selected $t_{\nu}$ for $\pi$. Here $\sigma^2_i$ is the estimated standard error of $x_i$, which we take to be known, and $t_{\nu}(\mu,\tau)$ denotes a three-parameter student $t$ distribution with density
\be
p(\Delta; \mu, \tau,\nu) = \frac{\Gamma(\frac{\nu+1}2)}{\Gamma(\frac{\nu}2) 
\sqrt{\pi \nu \tau^2}} \left(1 + \frac1{\nu} \left( \frac{x-\mu}{\tau} \right)^2 
\right)^{-\frac{\nu + 1}{2}},
\ee
a location-scale family with location $\mu$ and scale $\tau$.
%The unknown parameters of the model are therefore $\mu, \tau$, and $\nu$,

We fit the model in \eqref{eq:model} using Markov chain Monte Carlo (MCMC) 
implemented in the \texttt{Stan} environment with the \texttt{rstan} package 
\cite{rstan} for \texttt{R}. The data $(x_i,\sigma_i)$ for $i=1,\ldots,n$ are 
historical lift estimates and corresponding standard errors from A/B tests 
performed at eBay during the year 2016. Histograms of the $x_i$ and $\sigma_i$ 
are shown in the top and center panels of Figure \ref{fig:eda}. The $x_i$ are 
centered near zero and the distribution is apparently symmetric, but the tails 
are considerably heavier than Gaussian. The distribution of standard errors has 
a mode of approximately 0.3, with some values as large as 2.

\begin{figure}
\centering
\begin{tabular}{cc}
\includegraphics[width=0.5\textwidth]{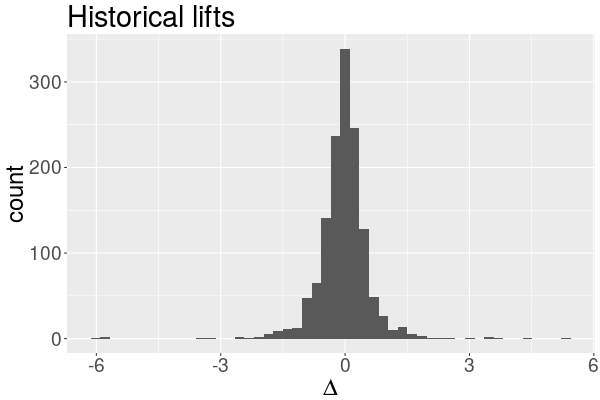} & \includegraphics[width=0.5\textwidth]{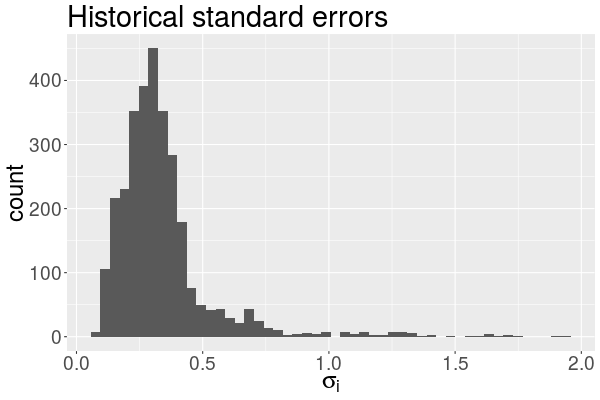} \\
\end{tabular}
\caption{Left: histogram of historical lifts. Right: histogram of historical standard errors.} \label{fig:eda}
\end{figure}

For priors, we put $\mu \sim \No(0,100)$, $\nu \sim \uniform(1.1,4)$, and 
$\tau \sim \Ca_+(0,1)$, the standard Cauchy distribution. The prior on $\mu$ is 
a default, rather vague, prior on a location parameter, and the prior on $\tau$ 
is recommended as a prior on variance components in hierarchical models by 
\citet{gelman2006prior} and  \citet{polson2012half}. The support of the uniform 
prior on $\nu$ is rather informative and was chosen based on preliminary 
analysis using Quantile-Quantile (Q-Q) plots. 

We run MCMC in \texttt{Stan} for 5,000 iterations, discarding 2,500 iterations 
as burn-in. This resulted in estimates of $\widehat \nu = 2.31$, $\widehat \mu 
= -0.02$, and $\widehat \tau = 0.18$; these estimates are the posterior 
mean estimate of these parameters.  
A Q-Q plot of the empirical quantiles of $x_i$ versus the fitted quantiles of 
$x_i$ in the model in \eqref{eq:model} with $(\nu,\mu,\tau) = (\widehat \nu, 
\widehat \mu, \widehat \tau)$ is shown in Figure \ref{fig:QQ}. 

In all of the analysis that follows, we use the posterior mean estimates $(\widehat \nu, \widehat \mu, \widehat \tau)$ to make a plug-in estimate of $\pi$. This is somewhat nonstandard in that we use a fully Bayesian procedure to estimate the parameters of $\pi$, but then follow an empirical Bayes approach to the rest of the analysis by fixing these parameters at the estimated posterior means. In other words, we use the Bayes machinery and MCMC simply to obtain lightly regularized point estimates of the parameters of $\pi$, in lieu of a more traditional non-regularized type II maximum likelihood approach. Experience with fitting $t$ distributions with unknown degrees of freedom to data suggests that some regulatization in our setting is wise.

% It is typically quite difficult to reliably estimate the degrees of freedom 
% parameter of a $t$ distribution, resulting in problems with MCMC convergence. 
% Thus, we estimated $\nu$ by making an initial guess, then fitting the model for 
% several different values of $\nu$ and choosing one that provided the best fit 
% assessed graphically using Quantile-Quantile (Q-Q) plots. 

\begin{figure}
\centering
 \includegraphics[width=0.6\textwidth]{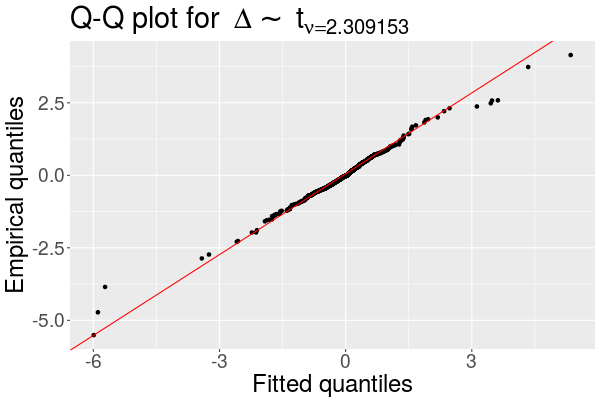}   
\caption{Q-Q plot of empirical quantiles of $x_i$ against fitted quantiles from the model in \eqref{eq:model} with $\Delta_i \stackrel{iid}{\sim} t_{2.35}(-0.02,0.18)$} \label{fig:QQ}
\end{figure}

We now approximate the Bayes-optimal threshold $\beta$ by numerically estimating the Bayes risk for the thresholding decision rule class $\delta(x) = \bone{x > \beta}$ for a grid of $\beta$ values. The optimal $\beta$ will most likely depend on $\sigma_i$, so we perform an initial analysis at $\sigma = 0.30$,\footnote{throughout, numbers reported in the text are rounded to two decimal places.} the median of the $\sigma_i$ in the data, then determine the optimal value of $\beta$ as a function of $\sigma$ for $0 < \sigma < 2$, yielding an optimal decision rule for any of the experiments conducted at eBay in 2016.

A plot of Bayes risk as a function of $\beta$ with $\sigma = 0.3$ is shown in 
Figure \ref{fig:BayesNormal}. The optimal value of $\beta$, which we will denote $\beta_{\opt}$, is $0.04$, which 
corresponds to a $p$-value threshold for the one-tailed test of $H_0$ in 
\eqref{eq:pvalue} of 0.45, with $p = 1-\Phi(\beta/\sigma)$, much less conservative than the default value of 
0.05 used in A/B testing. For most of the $\beta$ values considered, the Bayes 
risk is negative. This is significant, since the Bayes risk for the decision 
rule $\delta(x) = 0$ is zero; this corresponds to the limit as $\beta \to 
\infty$, so it must be the case that the risk converges to 0 as $\beta \to 
\infty$, consistent with the appearance of Figure \ref{fig:BayesNormal}. The 
optimal $\beta$ value of 0.04 is close to zero but not exactly zero. Recall that 
$\hat \tau = -0.02$, which is also the expectation of $\Delta_i$ in the fitted 
model. It is not a coincidence that this is slightly less than zero, and the 
optimal threshold is slightly greater than zero, as we will see theoretically 
for the case where $\pi$ is Gaussian in the next section. 

\begin{figure}
\centering
\includegraphics[width=0.6\textwidth]{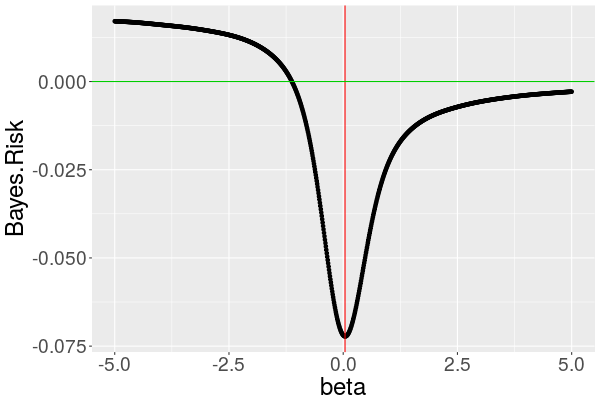}
\caption{Bayes risk vs $\beta$ for the model in \eqref{eq:model} using $\sigma_i = 0.3$ and the estimated values of $\nu, \mu, \tau$. The vertical red line shows the minimum value of $\beta$, the horizontal green line indicates zero risk.} \label{fig:BayesNormal}
\end{figure}

To obtain a value of $\beta_{\opt}$ for every case in the historical data, we need to estimate $\beta_{\opt}$ as a function of $\sigma$. To do this, we compute the Bayes risk over a two-dimensional fine grid of $(\beta, \sigma)$ values, then obtain $\beta_{\opt}(\sigma)$, the value of $\beta$ that minimizes the Bayes risk for every value of $\sigma$ considered. The results are shown in Figure \ref{fig:BetaSigma}. Clearly, the risk increases in $\sigma$, and $\beta_{\opt}(\sigma)$ also increases in $\sigma$ -- equivalently, the optimal $p$ value threshold decreases in $\sigma$ -- at least over the range of $\sigma$ values encountered in the data. Intuitively, this makes sense. If our observations $x_i$ of the true lift are very noisy, then we require a larger value of $x_i$ to have convincing evidence that the lift is indeed positive. 

\begin{figure}
\centering
\begin{tabular}{cc}
\includegraphics[width=0.5\textwidth]{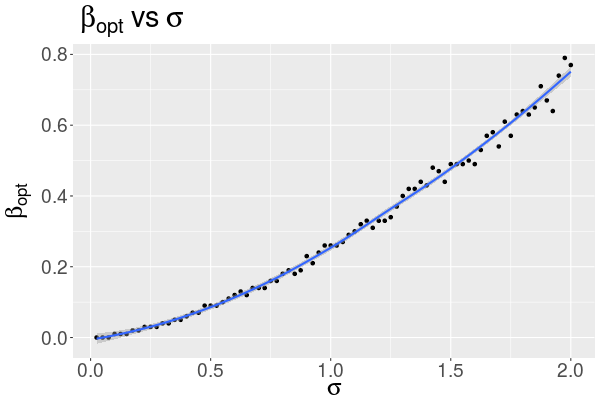} &
\includegraphics[width=0.5\textwidth]{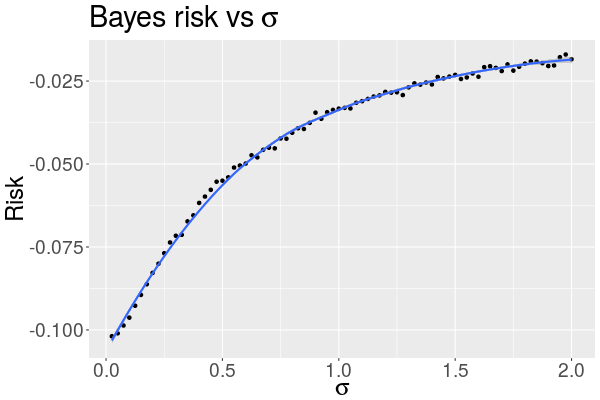}
\end{tabular}
\caption{Top: $\beta_{\opt}(\sigma)$. Bottom: Bayes risk vs $\sigma$ for 
$\beta_{\opt}(\sigma)$ for the model in \eqref{eq:model} at the estimated values 
of $\nu, \mu, \tau$. The blue lines show a local linear smooth.} 
\label{fig:BetaSigma}
\end{figure}

\section{Theoretical Results}
Considering still the case of the loss function in \eqref{eq:loss} with $F$ is a location-scale family CDF and frequentist risk given by \eqref{eq:LSRisk}, we now derive some simple results under the assumption that $\Delta \sim \pi(\Delta;\eta)$ with $\pi$ also a location-scale family, with both having densities symmetric about the location. This covers the examples in the previous section, and is arguably the most common type of model that would arise in applied settings. We have the following general result.
\begin{theorem}
Suppose $F$ is the distribution function of a location-scale family with a density $f$ that is symmetric about the location, and $\pi$ is the density of a location-scale family also symmetric about the location. Then if $\mathbb E_{\pi}(\Delta) = 0$, $\beta=0$ is a critical point of the Bayes risk.
\end{theorem}
\begin{proof}
We have
\be
\mathbb E_{\pi}[R(\Delta,\delta)] = \int -\Delta F\left( 
\frac{\Delta-\beta}{\sigma} \right) \frac1{\tau} \pi_0 \left( 
\frac{\Delta}{\tau} \right) d\Delta
\ee
where $\pi_0$ is the density of the standard member of the location-scale family. We have
\be
\frac{\partial}{\partial \beta} \mathbb E_{\pi}[R(\Delta,\delta)] &=  \int 
\frac{\partial}{\partial \beta}  -\Delta F\left( \frac{\Delta-\beta}{\sigma} 
\right) \frac1{\tau} \pi_0 \left( \frac{\Delta}{\tau} \right) d\Delta \\
&= \int \frac{\Delta}{\tau \sigma} f\left( \frac{\Delta-\beta}{\sigma} \right) 
\pi_0 \left( \frac{\Delta}{\tau} \right) d\Delta.
\ee
Observe that
\be
\frac{\partial}{\partial \beta} \mathbb E_{\pi}[R(\Delta,\delta)]\, 
\bigg|_{\beta = 0}  = \int \frac{\Delta}{\tau \sigma} f\left( 
\frac{\Delta}{\sigma} \right) \pi_0 \left( \frac{\Delta}{\tau} \right) d\Delta = 
0,
\ee
since the integrand is symmetric about zero, so $\beta = 0$ is always a critical point of the Bayes risk. 
\end{proof}
If this critical point is unique, it follows that if there exists a unique minimizer of the Bayes risk, it must be $\beta = 0$. Put another way, in a ``generic'' setup of this problem, when experiments have on average zero lift, then the optimal cutoff to use is $\beta = 0$, corresponding to a $p$-value cutoff of 0.5. We now show that for the case where both $F$ and $\pi$ are Gaussian, the optimal $\beta$ can be obtained in closed form for any values of the parameters of $F$ and $\pi$.

\begin{theorem} \label{thm:OptimalGauss}
Suppose $F$ is Gaussian, and $\pi$ is the density of a $\No(\mu,\tau^2)$ random variable. Then
\be
\beta = \frac{-\mu \sigma^2}{\tau^2}
\ee 
minimizes the Bayes risk. 
\end{theorem}
\begin{proof}
The Bayes risk is
\be
\bb E_{\pi}[R(\Delta,\delta)] = \int -\Delta \Phi\left( \frac{\Delta-\beta}{\sigma} \right) \phi \left( \frac{\Delta-\mu}{\tau} \right) d\Delta
\ee
with $\phi$ the standard Gaussian density, so
\be
\frac{\partial}{\partial \beta} \bb E_{\pi}[R(\Delta,\delta)] &= \int \frac{\Delta}{\sigma} \phi\left( \frac{\Delta-\beta}{\sigma} \right) \phi \left( \frac{\Delta-\mu}{\tau} \right) d\Delta \\
&= \frac{(\mu \sigma^2 + \beta \tau^2)\tau}{\sqrt{2\pi} (\sigma^2+\tau^2)^{3/2}} e^{-\frac{(\beta-\mu)^2}{2 (\sigma^2+\tau^2)}}. 
\ee
Setting equal to zero and solving gives the unique solution
\be
\beta = \frac{-\mu \sigma^2}{\tau^2},
\ee
and noting that
\be
&\frac{\partial}{\partial \beta} \frac{(\mu \sigma^2 + \beta \tau^2)\tau}{\sqrt{2\pi} (\sigma^2+\tau^2)^{3/2}} e^{-\frac{(\beta-\mu)^2}{2 (\sigma^2+\tau^2)}}  
%&=\frac{\mu \sigma^2(\mu-\beta) + (\sigma^2+\beta(\mu-\beta))\tau^2 +\tau^4}{\sqrt{2\pi} \sigma^2 \tau (\sigma^2+\tau^2)^2 \sqrt{\frac{1}{\sigma^2} + \frac{1}{\tau^2}}} e^{-\frac{(\beta-\mu)^2}{2(\sigma^2+\tau^2)}},
= e^{-\frac{(\beta-\mu)^2}{2 (\sigma^2 + \tau^2)}} \\ 
&\times \frac{\tau}{\sqrt{2 \pi} (\sigma^2 + \tau^2)^{3/2}} \left( \tau^2 - \frac{(\beta-\mu)(\mu \sigma^2 + \beta \tau^2)}{\sigma^2+\tau^2} \right)
\ee
which evaluated at $-\mu \sigma^2/\tau^2$ is
\be
\frac{\tau^3}{\sqrt{2 \pi} (\sigma^2+\tau^2)^{3/2}} e^{-\frac{\mu^2 (\sigma^2+\tau^2)}{2\tau^4}} > 0,
\ee
we conclude that $\beta = \frac{-\mu \sigma^2}{\tau^2}$ is the unique minimizer of the Bayes risk.
\end{proof}
This result is intuitive. The optimal cutoff is decreasing in the prior mean $\mu$. In other words, if most experiments tend to have large positive lifts, we become less conservative and accept proposed changes to the platform with weaker evidence that they are beneficial. The optimal threshold is also a linear function of the ratio $\sigma^2/\tau^2$ of the observation noise to the prior variance. Thus, when the observation noise is small relative to the variation in the true lifts, the optimal threshold is shrunk toward zero, meaning we accept an experiment with a small positive lift more readily than when the observation noise is large relative to $\tau^2$. This makes sense since in the former case we typically have smaller uncertainty about whether the true lift is positive than in the latter case.

% To validate these results empirically, we compute the Bayes optimal value of $\beta$ -- which we denote $\beta_{\opt}$ -- as a function of $\mu$ in the case where 
% \be
% x_i \sim \No(\Delta_i, \sigma^2_i), \quad \Delta_i \sim \No(\mu,\tau^2),
% \ee
% with $\tau = 0.18$ and $\sigma_i = 0.30$, the median observed value. The results 
% are shown in the top panel of Figure \ref{fig:OptimalBeta}, which also indicates 
% the line $-\mu \sigma^2/\tau^2$. Aside from some numerical problems when the 
% risk gets too close to zero, the optimal $\beta$ values are exactly on the line 
% predicted by the theory. \textbf{We will have to clean this figure up before 
% submitting, either by truncation or by improving the numerical accuracy.} 

Although we do not have a theoretical result for all $\mu, \tau$ for the model 
in \eqref{eq:model}, we can similarly evaluate the optimal $\beta$ empirically 
by fixing $\nu = 2.31$, $\sigma_i = 0.30$, and $\tau = 0.18$ and varying $\mu$. 
The resulting optimal $\beta$ value is shown in Figure 
\ref{fig:OptimalBeta}, along with the line $-\mu \sigma^2/\tau^2$ for comparison 
to the Gaussian case. Interestingly, in the region between -0.5 and 0.5, 
$\beta_{\opt}$ is a decreasing function of $\mu$, just as for the Gaussian, but the slope is
smaller than the $\sigma^2/\tau^2$ slope for the Gaussian prior. 
However, for larger or smaller values of $\mu$, $\beta_{\opt}$ moves back toward 
zero. 

An intuitive explanation of this phenomenon is that it is caused indirectly by the heavier tails of the prior relative to the (Gaussian) sampling model. When $|\mu| \ll \sigma$, $x$ and $\Delta$ will often have different signs, and thus the optimal threshold is an approximately linear function of the prior expectation of $\Delta$. We obtain relatively little information from $x$ and use more prior information in making the decision. When $|\mu| \gg \sigma$, an observed value of $x$ that is very far from $\mu$ most likely reflects a value of $\Delta$ that is very far from $\mu$, since outliers are much more common in the prior than in the sampling model. Thus, a threshold closer to zero makes sense, since variation in the prior swamps the observation noise. This is why the value of $\beta_{\opt}$ flattens around the value $|\mu| = 0.3 = \sigma$ and then moves back toward zero in Figure \ref{fig:OptimalBeta}.  %\textbf{I still have aspirations of improving this paragraph; suggestions welcome}.

\begin{figure}
\centering
\includegraphics[width=0.5\textwidth]{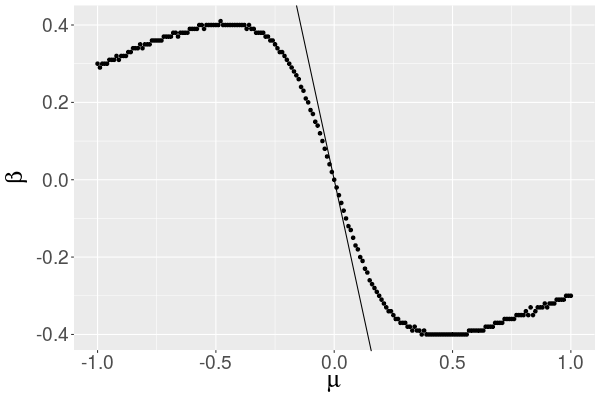} 
\caption{Bayes optimal $\beta$ vs $\mu$ with $\nu=2.31$, $\sigma=0.3$, and 
$\tau=0.18$ for model in \eqref{eq:model}.} \label{fig:OptimalBeta}
\end{figure}

%Since the $t$ distribution is a scale mixture of normals, we can re-use some of this calculation to obtain expressions for the optimal $\beta$ when the likelihood is Gaussian and the prior is $t$, but the minimizer does not appear to be analytically available, and rather is given by an integral expression.

\section{Hierarchical Structure of Experiments} \label{sec:Hierarchical}
We have until now ignored the fact that some experiments may be more related than others, opting for a simple hierarchical model. Often, if a feature is developed and is not selected after the first A/B test, the team that developed the feature will modify the algorithm and then re-test. This gives rise to sequences of closely related tests. If we treat all the observations $x_i$ as having means $\Delta_i$ that are iid from the random effect distribution, we are ignoring this structure in the data. 

The practice of modifying and re-testing may offer a partial explanation for the use of $0.05$ as a $p$-value threshold, which our analysis suggests is much too conservative when performing single, exchangeable experiments. Recall that we computed a Bayes-optimal threshold $\beta = 0.04$, corresponding to a one-tailed $p$ value of $0.45$. If instead of a single experiment yielding a single noisy measurement $x_i$ of the true lift $\Delta_i$, we performed $n_i$ experiments yielding $n_i$ noisy measurements of $\Delta_i$, a simple Bonferroni correction would indicate performing each test by thresholding the $p$ value at $0.45/n_i$, which is $0.05$ for $n_i = 9$. Figure \ref{fig:ecdf} shows the empirical distribution function (ECDF) for the number of replicate tests of each experiment conducted by eBay in 2016. The 95th percentile is 6, and the 99th percentile is 9. Thus, if we translate the optimal threshold for single tests into a Bonferroni-corrected $p$-value threshold for multiple tests, a threshold of 0.05 would be appropriate to uniformly control multiplicity for 99 percent of the features tested at eBay.

\begin{figure}
\centering
 \includegraphics[width=0.6\textwidth]{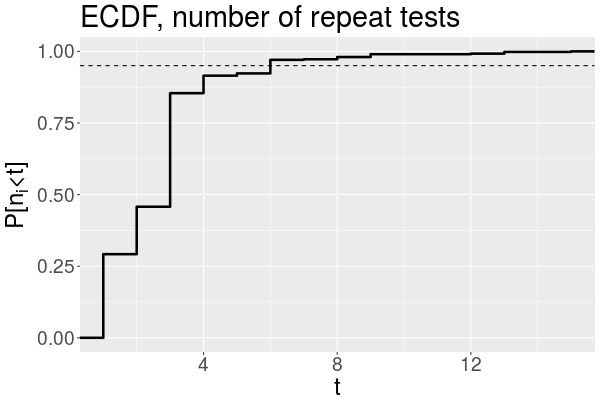}
 \caption{Empirical CDF of the number of experiments $n_i$ done for feature $i$ in 2016 across all of eBay. } \label{fig:ecdf}
\end{figure}

This simple analysis is unsatisfactory because it lapses back into the testing framework that we have sought to avoid. To extend the decision-theoretic approach to the case of repeated observations, we now analyze the Bayes risk in this setting. For simplicity, we consider the case where modifications to a feature after the first experiment have no affect on the true lift. %In reality, the data suggest a small but significant improvement resulting from modifications after the first test, but an analysis of .%, which we address in the discussion.

If modifications have no effect, then instead of observing one noisy realization of the lift $x \sim F(x \mid \Delta)$, we observe $n_i$ noisy data points $x_{i1},\ldots,x_{i n_i}$ for each feature. Our decision rule $\delta(x)$ is now a function of $n_i$ many observations, so $\delta : \bb R^{n_i} \to \{0,1\}$. If, as before, $x_{ij} \stackrel{iid}{\sim} F(x;\Delta_i,\sigma^2_{ij})$, then 
\be \label{eq:WtdMean}
Y_i := \sum_{j=1}^{n_i} \frac{X_{ij}}{\sigma^2_{ij}} \stackrel{.}{\sim} \No\left( \frac{\Delta_i}{S_i}, \frac{1}{S_i} \right),
\ee 
where $X \stackrel{.}{\sim} \mathcal L$ indicates that the random variable $X$ approximately follows the law $\mathcal L$, and 
\be
S_i \equiv \left( \sum_{j} \frac1{\sigma^2_{ij}} \right)^{-1},
\ee
$n_i$ times the harmonic mean of the observation variances. The motivation to weight by the inverse variances will soon become apparent. Thus, if $\delta$ corresponds to thresholding the inverse variance-weighted sum, so that
\be
\delta(x) = \bone{y_i > \beta},
\ee 
then the risk is
\be
R(\Delta,\delta) &= -\Delta \int \bone{ y > \beta} F(y \mid \Delta, \sigma_1,\ldots,\sigma_n) dy \\
&= -\Delta \P[ Y > \beta] \\
&= -\Delta \P\left[ \frac{Y-\Delta/S}{1/\sqrt{S}} > \frac{\beta - \Delta/S}{1/\sqrt{S}} \right] \\
%&= -\Delta \left[ 1-\Phi\left( \frac{\beta - \Delta}{S_{\sigma}/\sqrt{n}} \right) \right] \\
&= -\Delta \Phi\left( \frac{\Delta/S-\beta}{1/\sqrt{S}} \right),
\ee
where we have dropped subscripts above to simplify notation. Now we derive the Bayes risk in the case where $\Delta \sim \No(\mu,\tau^2)$.
\begin{theorem}
Suppose $\delta(x) = \bone{y>\beta}$, and $\pi$ is the density of a $\No(\mu,\tau^2)$ random variable. Then
\be
\beta = \frac{-\mu}{\tau^2}
\ee 
minimizes the Bayes risk. 
\end{theorem}
\begin{proof}
The Bayes risk is
\be
\bb E_{\pi}[R(\Delta,\delta)] = \int -\Delta \Phi\left( \frac{\Delta/S-\beta}{1/\sqrt{S}} \right) \phi \left( \frac{\Delta-\mu}{\tau} \right) d\Delta
\ee
with $\phi$ the standard Gaussian density, so
\be
\frac{\partial}{\partial \beta} \bb E_{\pi}[R(\Delta,\delta)] &= \int \sqrt{S} \Delta \phi\left( \frac{\Delta/S-\beta}{1/\sqrt{S}} \right)\phi \left( \frac{\Delta-\mu}{\tau} \right) d\Delta \\
&= \frac{S^2 \tau (\mu + \beta \tau^2)}{\sqrt{2 \pi} (S + \tau^2)^{3/2}} e^{-\frac{(S\beta - \mu)^2}{2 (S + \tau^2)}}. 
\ee
Setting equal to zero and solving gives the unique solution
\be
\beta = \frac{-\mu}{\tau^2}.
\ee
The remaining details are similar to the proof of Theorem \ref{thm:OptimalGauss} and are omitted.
%and noting that
%\be
%\frac{\partial}{\partial \beta} \frac{\mu + \beta \tau^2}{\sqrt{2\pi} (1+\tau^2)^{3/2}} e^{-\frac{\beta-\mu)^2}{2 (1+\tau^2)}} &= \frac{\mu(\mu-\beta) + (1+\beta(\mu-\beta))\tau^2 +\tau^4}{\sqrt{2\pi} (1+\tau^2)^{5/2}} e^{-\frac{(\beta-\mu)^2}{2(1+\tau^2)}},
%\ee
%which evaluated at $-\mu/\tau^2$ is
%\be
%\frac{\tau^2}{\sqrt{2 \pi} (1+\tau^2)^{3/2}} e^{-\frac{\mu^2 (1+\tau^2)}{2\tau^4}} > 0,
%\ee
%we conclude that $\beta = \frac{-\mu}{\tau^2}$ is the unique minimizer of the Bayes risk.
\end{proof}
Thus, thresholding the inverse-variance weighted average yields an optimal threshold that is independent of the variances. Notice that when $n=1$, $Y = X/\sigma^2$, and we recover the optimal threshold for $X$ in Theorem \ref{thm:OptimalGauss}. 

%The obvious difference here is that as $n$ increases the threshold gets closer to zero. This makes sense -- as we gather more data, we are increasingly certain about the value of $\Delta$, so we tend to ship if the estimate is slightly positive and not ship if it is slightly negative.

In reality, our decision problem is typically whether to ship the $n$th version of the feature having already collected $n-1$ noisy observations of the lifts of previous versions. If $\Delta \sim \No(\mu,\tau^2)$ then we have
\be
p(\Delta \mid x_{1:n},\sigma_{1:n}) &\propto p(\Delta) \prod_{j=1}^{n} p(x_i \mid \Delta,\sigma_{j}^2) \\
&\propto \phi\left( \frac{\Delta-\mu}{\tau} \right) \prod_{j=1}^{n} \phi\left( \frac{x_j-\Delta}{\sigma_{j}} \right)
\ee
so
\be
\Delta & \mid x_{1},\ldots,x_{n-1}, \mu, \tau \sim \No\left( s^{-1} m, s^{-1} \right) \\
s& = \left( \frac1{\tau^2} + \sum_{j=1}^{n-1} \frac{1}{\sigma^2_j} \right) \quad m = \left( \frac{\mu}{\tau^2} + \sum_{j=1}^{n-1} \frac{x_j}{\sigma^2_j} \right),
\ee
and we can immediately apply Theorem \ref{thm:OptimalGauss} to obtain
\be
\beta_{\opt} &= \frac{-\left( \frac1{\tau^2} +\sum_{j=1}^{n-1} \frac{1}{\sigma^2_j} \right)^{-1} \left( \frac{\mu}{\tau^2} + \sum_{j=1}^{n-1} \frac{x_j}{\sigma^2_j} \right) \sigma^2_n}{\left( \frac1{\tau^2} + \sum_{j=1}^{n-1} \frac{1}{\sigma^2_j} \right)^{-1}} \\
&=  -\left( \frac{\mu}{\tau^2} + \sum_{j=1}^{n-1} \frac{x_j}{\sigma^2_j} \right) \sigma^2_n \\
&= \frac{-\mu \sigma^2_n}{\tau^2} - \sigma^2_n \sum_{j=1}^{n-1} \frac{x_j}{\sigma^2_j},
\ee
which is essentially a sum of the optimal threshold for one experiment and the weighted sum of the observations for the past $n-1$ experiments. This turns out to be identical for the optimal threshold if we consider all of the experiments jointly and threshold the inverse variance weighed sum, since
\be
\sum_{j=1}^n \frac{x_j}{\sigma^2_j} = \frac{-\mu}{\tau^2} \Longleftrightarrow x_n = \frac{-\mu \sigma^2_n}{\tau^2} -\sigma^2_n \sum_{j=1}^{n-1}\frac{x_j}{\sigma^2_j},
\ee
so that the only difference is operational: we either apply a threshold to $y$, or we apply a threshold that depends on the weighted sum of the previous $n-1$ experiments only the latest experiment $x_n$. Thus, as data about the lift of a particular feature accumulates, we become more certain about the true value of the lift, and a larger effect size is necessary to convince us that the feature has the opposite effect that previous tests indicated. 

We cannot perform this calculation analytically for the case where the $\Delta$ follow a $t$ distribution, but we can compute the optimal threshold empirically as before. To do this, we fit the model 
\be \label{eq:mod2}
x_{ij} &\sim \No(\Delta_i,\sigma^2_{ij}) \\
\Delta_i &\sim t_{\nu}(\mu,\tau),
\ee
and compute the Bayes risk for thresholding $y_i$ as defined in
\eqref{eq:WtdMean} on a grid of $\beta$ values. The results are shown in Figure 
\ref{fig:OptBetaHier}. The optimal value of $\beta$ is $0.31$, which 
corresponds to a threshold for $x_i$ for a single experiment with $\sigma^2_i = 
0.3$ of $\beta \sigma^2_i = 0.09$ and $p$-value threshold of $0.46$, very 
similar to the optimal threshold in the model where all experiments were 
assigned their own lifts. %The small discrepancy reflects differences in the 
%estimated values of $\tau$, $\nu$, and $\mu$ for the model in \eqref{eq:mod2} 
%compared to the model in \eqref{eq:model}.   

\begin{figure}
\centering
 \includegraphics[width=0.6\textwidth]{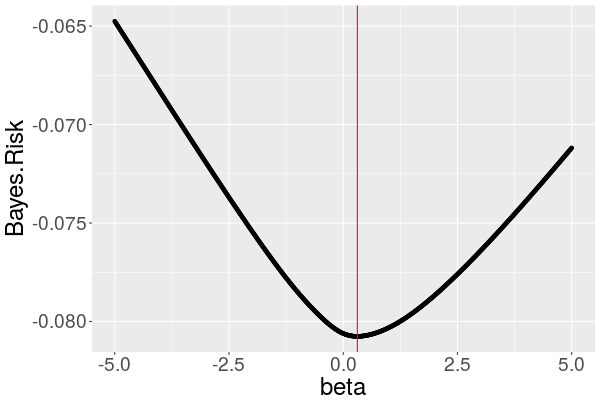}
 \caption{Bayes risk vs $\beta$ for thresholding decision rule $\delta(x) = \bone{y>\beta}$ computed on model in \eqref{eq:mod2}. Horizontal green line at zero risk, vertical red line at optimal $\beta$. } \label{fig:OptBetaHier}
\end{figure}

\section{Discussion}
Traditional A/B testing has treated the decision on whether to ship a feature as a hypothesis test, requiring research and development teams to make an arbitrary choice of a $p$ value threshold at which to adopt a new feature. The decision theoretic approach we outline here has the potential to automate this choice by using data on previous tests to inform about an optimal threshold via Bayesian analysis. We have not considered here the case where modifications to a feature effect the lift, though we do find some evidence of a small average improvement due to modifications at eBay. Extending the decision theoretic analysis of optimal thresholds to this more complicated setting is an interesting extension of the current work.

\subsubsection*{Acknowledgements}

The authors thank Kristian Lum for suggesting consideration of the $n$th experiment after observing $n-1$ experiments in section 5, and generally helpful discussions. We thank David Dunson for useful comments on a draft.

%\subsubsection*{References}

\bibliographystyle{abbrvnat}
\bibliography{abtest-dt}

%References follow the acknowledgements.  Use an unnumbered third level
%heading for the references section.  Any choice of citation style is
%acceptable as long as you are consistent.  Please use the same font
%size for references as for the body of the paper---remember that
%references do not count against your page length total.
%
%J.~Alspector, B.~Gupta, and R.~B.~Allen (1989). Performance of a
%stochastic learning microchip.  In D. S. Touretzky (ed.), {\it
%  Advances in Neural Information Processing Systems 1}, 748-760.  San
%Mateo, Calif.: Morgan Kaufmann.
%
%F.~Rosenblatt (1962). {\it Principles of Neurodynamics.} Washington,
%D.C.: Spartan Books.
%
%G.~Tesauro (1989). Neurogammon wins computer Olympiad.  {\it Neural
%  Computation} {\bf 1}(3):321-323.

\end{document}